\documentclass[a4paper,reqno,10pt,oneside]{amsart}

\usepackage{amsfonts,amssymb,latexsym,xspace,epsfig,graphics,color}
\usepackage{amsmath,enumerate,stmaryrd,xy}
\usepackage{bbm}
\usepackage{natbib}
\numberwithin{figure}{section}
\usepackage{subfigure}
\usepackage{dsfont}
\usepackage{a4wide}
\usepackage{enumitem}   

\usepackage{tikz}
\usepackage{tkz-berge}
\usetikzlibrary{arrows,snakes,backgrounds,trees}
\usepackage{stmaryrd}

\newtheorem{theorem}{Theorem}[section]       
\newtheorem{conjecture}{Conjecture}                                        
\newtheorem{prop}[theorem]{Proposition}

\newtheorem{defn}[theorem]{Definition}
\newtheorem{remark}{Remark}[section]
\newtheorem{exa}{Example}[section]

\def\T+{{\mathbb T_d^+}}

\def\A{\mathcal{A}}

\def\rank {\mathop {\rm rank}\nolimits}

\def\Im {\mathop {\rm Im}\nolimits}
\def\N{{\mathcal N}}

\def\supp {\mathop {\rm supp}\nolimits}
\def\aut {\mathop {\rm Aut}\nolimits}

\bibdata{prob}
\bibstyle{alpha} 

\title[]{On the isomorphisms between evolution algebras of graphs and random walks}

\author{Paula Cadavid, Mary Luz Rodi\~no Montoya and Pablo M. Rodr\'iguez}

\date{}

\address{
\newline
Paula Cadavid, Pablo M. Rodriguez
\newline
Instituto de Ci\^encias Matem\'aticas e de Computa\c{c}\~ao, Universidade de S\~ao Paulo
\newline  
Caixa Postal 668, 13560-970 S\~ao Carlos, SP, Brazil
\newline
e-mails: pacadavid@usp.br, pablor@icmc.usp.br
\newline
\newline
Mary Luz Rodi\~no Montoya
\newline
Instituto de Matem\'aticas - Universidad de Antioquia 
\newline
Calle 67 N$^{\circ}$ 53-108, Medell\'in, Colombia
\newline 
e-mail: mary.rodino@udea.edu.co 
}


\begin{document}
  
\maketitle

\begin{abstract}
Evolution algebras are non-associative algebras inspired from biological phenomena, with applications to or connections with different mathematical fields. There are two natural ways to define an evolution algebra associated to a given graph. While one takes into account only the adjacencies of the graph, the other includes probabilities related to the symmetric random walk on the same graph. In this work we state new properties related to the relation between these algebras, which is one of the open problems in the interplay between evolution algebras and graphs. On the one hand, we show that for any graph both algebras are strongly isotopic. On the other hand, we provide conditions under which these algebras are or are not isomorphic. For the case of {\color{black}finite} non-singular graphs we provide a complete description of the problem, while for the case of {\color{black}finite} singular graphs we state a conjecture supported by examples and partial results. {\color{black}The case of graphs with an infinite number of vertices is also discussed.} As a sideline {\color{black}of our work}, we revisit a result existing in the literature about the identification of the automorphism group of an evolution algebra, and we give an improved version of it.
\end{abstract}


\section{Introduction}

In this paper we study evolution algebras, which are a new type of non-associative algebras. These algebras were introduced around ten years ago by Tian \cite{tian} and were motivated by evolution laws of genetics. With this application in mind, if one think in alleles as generators of algebras, then reproduction in genetics is represented by multiplication in algebra. The best general reference of the subject is \cite{tian}, where the reader can found a review of preliminary definitions and properties, connections with other fields of mathematics, and a list of interesting open problems some of which remain unsolved so far. We refer the reader also to \cite{tian2} for an update of open problems in the Theory of Evolution Algebras, and to \cite{PMP}-\cite{Falcon/Falcon/Nunez/2017} and references therein for an overview of recent results on this topic. Formally, an evolution algebra is defined as follows.

\begin{defn}\label{def:evolalg}
Let $\A:=(\A,\cdot\,)$ be an algebra over a field $\mathbb{K}$. We say that $\A$ is an evolution algebra if it admits a countable basis $S:=\{e_1,e_2,\ldots , e_n,\ldots\}$, such that

\begin{equation}\label{eq:ea}
\begin{array}{ll}
e_i \cdot e_i =\displaystyle \sum_{k} c_{ik} e_k,&\text{for any }i,\\[.3cm]
e_i \cdot e_j =0,&\text{if }i\neq j.
\end{array}
\end{equation} 

\smallskip
\noindent
The scalars $c_{ik}\in \mathbb{K}$ are called the structure constants of $\mathcal{A}$ relative to $S$.
\end{defn}

\smallskip
A basis $S$ satisfying \eqref{eq:ea} is called natural basis of $\mathcal{A}$. $\mathcal{A}$ is real if $\mathbb{K}=\mathbb{R}$, and it is nonnegative if it is real and the structure constants $c_{ik}$ are nonnegative. In what follows, we always assume that $\mathcal{A}$ is real. In addition, if  $0\leq c_{ik}\leq 1$, and 

$$\sum_{k=1}^{\infty}c_{ik}=1,$$ 

\noindent
for any $i$, then $\A$ is called a Markov evolution algebra. In this case, there is an interesting correspondence between the algebra $\A$ and a discrete time Markov chain $(X_n)_{n\geq 0}$ with states space $\{x_1,x_2,\ldots,x_n,\ldots\}$ and transition probabilities given by:
\begin{equation}\label{eq:tranprob}
\nonumber c_{ik}:=\mathbb{P}(X_{n+1}=x_k|X_{n}=x_i),\end{equation}
for $i,k\in \mathbb{N}^*$, and for any $n\in \mathbb{N}$, where $\mathbb{N}^*:=\mathbb{N}\setminus \{0\}$. For the sake of completeness we remind the reader that a discrete-time Markov chain is a sequence of random variables $X_0, X_1, X_2, \ldots, X_n, \ldots$, defined on the same probability space $(\Omega,\mathcal{B},\mathbb{P})$, taking values on the same set $\mathcal{X}$, and such that the Markovian property is satisfied, i.e., for any set of values $\{i_0, \ldots, i_{n-1},x_i, x_{k}\} \subset \mathcal{X}$, and any $n\in \mathbb{N}$, it holds
$$\mathbb{P}(X_{n+1}=x_k|X_0 = i_0, \ldots, X_{n-1}=i_{n-1}, X_{n}=x_i)=\mathbb{P}(X_{n+1}=x_k|X_{n}=x_i).$$ 
Thus defined, in the correspondence between the evolution algebra $\A$ and the Markov chain $(X_n)_{n\geq 0}$ what we have is each state of $\mathcal{X}$ identified with a generator of $S$. For more details about the formulation and properties of Markov chains we refer the reader to \cite{karlin/taylor,ross}. In addition, for a review of results related to the connection between Markov chains and evolution algebras we suggest \cite[Chapter 4]{tian}.

\smallskip
In this work we are interested in studying evolution algebras related to graphs in a sense to be specified later. This interplay, i.e. evolution algebras and graphs, has attained the attention of many researchers in recent years. For a review of recent results, see for instance \cite{PMP,PMP2,camacho/gomez/omirov/turdibaev/2013,Elduque/Labra/2015,nunez/2014, nunez/2013}, and references therein. The rest of the section is subdivided into two parts. In the first one we review some of the standard notation of Graph Theory, while in the last one we give the definition of two different evolution algebras associated to a given graph. One of the open questions of the Theory of Evolution Algebras is to understand the relation between both induced algebras. The purpose of this paper is to advance in this question. 


\subsection{Basic notation of Graph Theory}

A graph $G$ with $n$ vertices is a pair $(V,E)$ where $V:=\{1,\ldots,n\}$ is the set of vertices and $E:=\{(i,j)\in V\times V:i\leq j\}$ is the set of edges. If $(i,j)\in E$ or $(j,i)\in E$ we say that $i$ and $j$ are neighbors; we denote the set of neighbors of vertex $i$ by $\mathcal{N}(i)$ and the cardinality of this set by $\deg(i)$. Our definitions as well as our results, except when indicated, also hold for graphs with an infinite number of vertices, i.e. $V$ is a countable set and $|V|=\infty$. In that case we assume as an additional condition for the graph to be locally finite, i.e. $\deg(i)< \infty$ for any $i\in V$. In general, if $U\subseteq V$, we denote $\mathcal{N}(U) :=\{ j \in V : j\in \mathcal{N}(i) \text{ for some } i\in U\}$. We say that $G$ is a $d$-regular graph if $\deg(i) = d$ for any $i\in V$ and some positive integer $d$. We say that $G$ is a bipartite graph if its vertices can be divided into two disjoint sets, $V_1$ and $V_2$, such that every edge connects a vertex in $V_1$ to one in $V_2$. If $V_1$ has $m$ vertices, $V_2$ has $n$ vertices and every possible edge that could connect vertices in different subsets is part of the graph we call $G$ a complete bipartite graph and denote it by $K_{m,n}$. Moreover, we say that $G$ is a biregular graph if it is a bipartite graph $G=(V_1,V_2,E)$ for which every two vertices on the same side of the given bipartition have the same degree as each other. In this case, if the degree of the vertices in $V_1$ is $d_1$ and the degree of the vertices in $V_2$ is $d_2$, then we say that the graph is $(d_{1},d_2)$-biregular  (see Fig. \ref{fig:bipartite}). We notice that the family of biregular graphs includes any finite graph which may be seen as a bipartite graph with partitions $V_1$ and $V_2$ of sizes $m$ and $n$ respectively, for $m,n\geq 1$, such that $\deg(i) =d_1$ if $i\in V_1$, $\deg(i) =d_2$ if $i\in V_2$, where $d_1,d_2 \in \mathbb{N}$ satisfy $m\, d_1 = n\, d_2$, see Fig. \ref{fig:bipartite}(b). In addition, the class of biregular graphs includes some families of infinite graphs like $2$-periodic trees (see Fig. \ref{fig:tree}(a)) and $\mathbb{Z}^2$-periodic graphs with hexagonal lattice (see Fig. \ref{fig:tree}(b)).

\begin{figure}[h!]
\begin{center}

\subfigure[][The set of vertices may be partitioned into the two subsets $V_1=\{5,6,7,8,9,10\}$ and $V_2=\{1,2,3,4\}$, with degrees $2$ and $3$, resp.]{

\begin{tikzpicture}[scale=0.8]

\draw (-1.5,3) -- (-3,1.5)--(-1.5,0)--(-1.5,1.5)--(-1.5,3)--(0,3)--(1.5,3)--(3,1.5)--(1.5,0)--(0,0)--(-1.5,0);
\draw (1.5,3)--(1.5,0);

\filldraw [black] (0,0) circle (2pt);
\draw (0,-0.2) node[below,font=\footnotesize] {$7$};
\filldraw [black] (1.5,0) circle (2pt);
\draw (1.5,-0.2) node[below,font=\footnotesize] {$3$};
\filldraw [black] (-1.5,0) circle (2pt);
\draw (-1.5,-0.2) node[below,font=\footnotesize] {$4$};
\filldraw [black] (1.5,1.5) circle (2pt);
\draw (1.8,2) node[font=\footnotesize] {$6$};
\filldraw [black] (-1.5,1.5) circle (2pt);
\draw (-1.8,1.5) node[font=\footnotesize] {$8$};
\filldraw [black] (1.5,3) circle (2pt);
\draw (1.5,3.2) node[above,font=\footnotesize] {$2$};
\filldraw [black] (-1.5,3) circle (2pt);
\draw (-1.5,3.2) node[above,font=\footnotesize] {$1$};
\filldraw [black] (0,3) circle (2pt);
\draw (0,3.2) node[above,font=\footnotesize] {$5$};
\filldraw [black] (3,1.5) circle (2pt);
\draw (3.3,1.5) node[font=\footnotesize] {$9$};
\filldraw [black] (-3,1.5) circle (2pt);
\draw (-3.3,1.5) node[font=\footnotesize] {$10$};

\filldraw [white] (0,-1) circle (2pt);

    \end{tikzpicture}
}\qquad\qquad\qquad  \subfigure[][Representation of the finite $(2,3)$-biregular graph as a bipartite graph.]{

\begin{tikzpicture}[scale=0.8]

\draw (-2,1)--(2,2);
\draw (-2,1)--(2,-1);
\draw (-2,1)--(2,-3);

\draw (-2,0)--(2,2);
\draw (-2,0)--(2,1);
\draw (-2,0)--(2,-2);

\draw (-2,-1)--(2,1);
\draw (-2,-1)--(2,0);
\draw (-2,-1)--(2,-2);

\draw (-2,-2)--(2,0);
\draw (-2,-2)--(2,-1);
\draw (-2,-2)--(2,-3);

\filldraw [white] (4,0) circle (2pt);
\filldraw [white] (-3.5,0) circle (2pt);

\filldraw [black] (-2,-2) circle (2pt);
\draw (-2.3,-2) node[font=\footnotesize] {$4$};
\filldraw [black] (-2,1) circle (2pt);
\draw (-2.3,1) node[font=\footnotesize] {$1$};
\filldraw [black] (-2,-1) circle (2pt);
\draw (-2.3,-1) node[font=\footnotesize] {$3$};
\filldraw [black] (-2,0) circle (2pt);
\draw (-2.3,0) node[font=\footnotesize] {$2$};

\filldraw [black] (2,2) circle (2pt);
\draw (2.3,2) node[font=\footnotesize] {$5$};
\filldraw [black] (2,1) circle (2pt);
\draw (2.3,1) node[font=\footnotesize] {$6$};
\filldraw [black] (2,0) circle (2pt);
\draw (2.3,0) node[font=\footnotesize] {$7$};
\filldraw [black] (2,-1) circle (2pt);
\draw (2.3,-1) node[font=\footnotesize] {$8$};
\filldraw [black] (2,-2) circle (2pt);
\draw (2.3,-2) node[font=\footnotesize] {$9$};
\filldraw [black] (2,-3) circle (2pt);
\draw (2.3,-3) node[font=\footnotesize] {$10$};

\end{tikzpicture}
}

\end{center}
\caption{A $(2,3)$-biregular graph with $10$ vertices.}\label{fig:bipartite}
\end{figure}
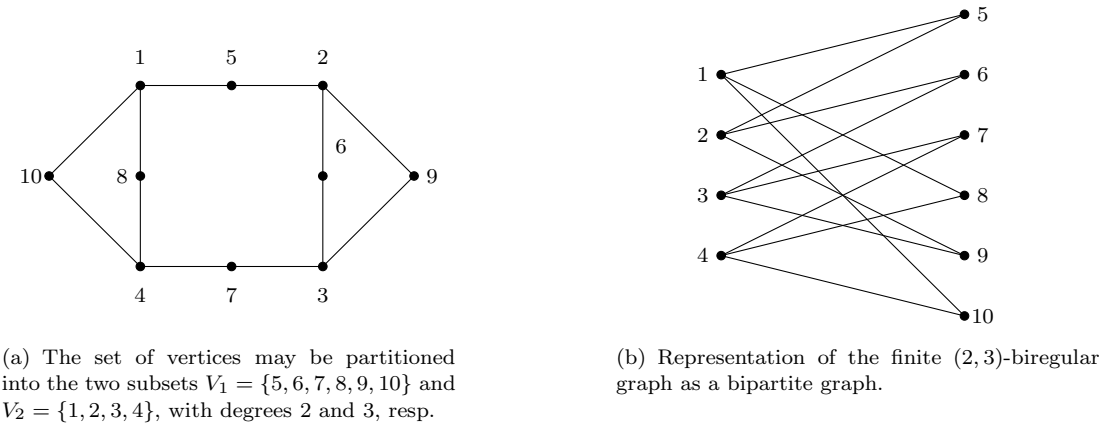

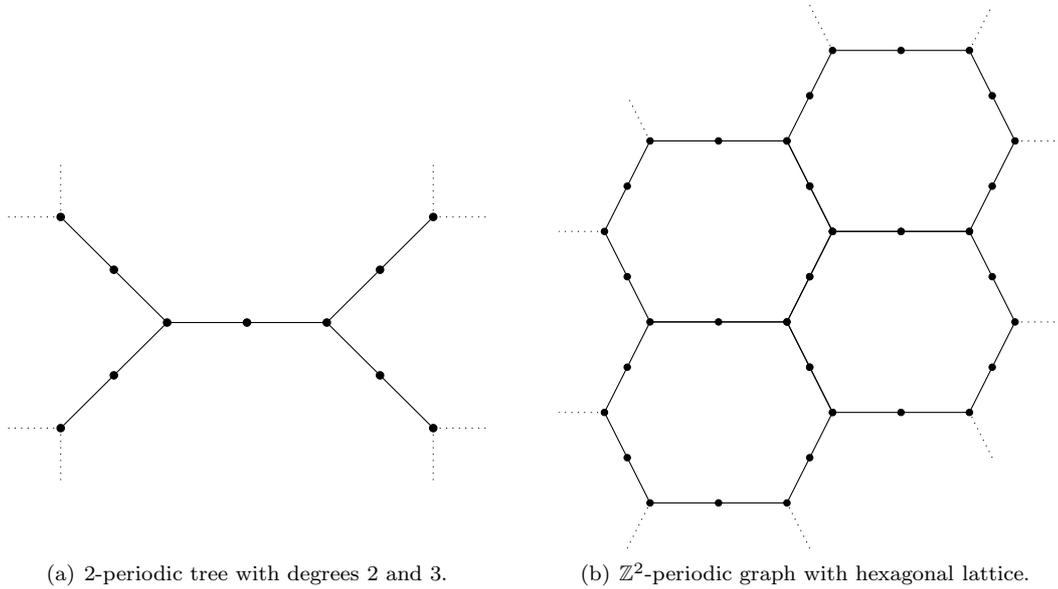
\begin{figure}[h!]
\begin{center}

\subfigure[][$2$-periodic tree with degrees $2$ and $3$.]{

\begin{tikzpicture}[scale=0.7]

\draw (-1.5,0)--(1.5,0);
\draw (-1.5,0)--(-3.5,2);
\draw (-1.5,0)--(-3.5,-2);
\draw (1.5,0)--(3.5,2);
\draw (1.5,0)--(3.5,-2);

\draw [dotted] (3.5,2)--(3.5,3);
\draw [dotted] (3.5,2)--(4.5,2);
\draw [dotted] (3.5,-2)--(3.5,-3);
\draw [dotted] (3.5,-2)--(4.5,-2);
\draw [dotted] (-3.5,2)--(-3.5,3);
\draw [dotted] (-3.5,2)--(-4.5,2);
\draw [dotted] (-3.5,-2)--(-3.5,-3);
\draw [dotted] (-3.5,-2)--(-4.5,-2);

\filldraw [black] (0,0) circle (2pt);
\filldraw [black] (1.5,0) circle (2pt);
\filldraw [black] (-1.5,0) circle (2pt);
\filldraw [black] (2.5,1) circle (2pt);
\filldraw [black] (2.5,-1) circle (2pt);
\filldraw [black] (-2.5,1) circle (2pt);
\filldraw [black] (-2.5,-1) circle (2pt);

\filldraw [black] (3.5,2) circle (2pt);
\filldraw [black] (3.5,-2) circle (2pt);
\filldraw [black] (-3.5,2) circle (2pt);
\filldraw [black] (-3.5,-2) circle (2pt);

\filldraw [white] (0,-4.2) circle (2pt);

\end{tikzpicture}
}\qquad  \subfigure[][$\mathbb{Z}^2$-periodic graph with hexagonal lattice.]{
\begin{tikzpicture}[scale=0.6]

\draw (-1.5,0)--(1.5,0)--(2.5,2)--(1.5,4)--(-1.5,4)--(-2.5,2)--(-1.5,0);

\filldraw [black] (0,0) circle (2pt);
\filldraw [black] (1.5,0) circle (2pt);
\filldraw [black] (-1.5,0) circle (2pt);
\filldraw [black] (1.5,4) circle (2pt);
\filldraw [black] (-1.5,4) circle (2pt);
\filldraw [black] (0,4) circle (2pt);

\filldraw [black] (-2.5,2) circle (2pt);
\filldraw [black] (2.5,2) circle (2pt);

\filldraw [black] (-2,3) circle (2pt);
\filldraw [black] (2,3) circle (2pt);

\filldraw [black] (-2,1) circle (2pt);
\filldraw [black] (2,1) circle (2pt);

\draw (-1.5,4)--(1.5,4)--(2.5,6)--(1.5,8)--(-1.5,8)--(-2.5,6)--(-1.5,4);

\filldraw [black] (0,8) circle (2pt);
\filldraw [black] (1.5,8) circle (2pt);
\filldraw [black] (-1.5,8) circle (2pt);

\filldraw [black] (-2.5,6) circle (2pt);
\filldraw [black] (2.5,6) circle (2pt);

\filldraw [black] (-2,7) circle (2pt);
\filldraw [black] (2,7) circle (2pt);

\filldraw [black] (-2,5) circle (2pt);
\filldraw [black] (2,5) circle (2pt);

\draw (2.5,2)--(5.5,2)--(6.5,4)--(5.5,6)--(2.5,6)--(1.5,4)--(2.5,2);

\filldraw [black] (4,2) circle (2pt);
\filldraw [black] (5.5,2) circle (2pt);
\filldraw [black] (2.5,2) circle (2pt);
\filldraw [black] (5.5,6) circle (2pt);
\filldraw [black] (2.5,6) circle (2pt);
\filldraw [black] (4,6) circle (2pt);

\filldraw [black] (1.5,4) circle (2pt);
\filldraw [black] (6.5,4) circle (2pt);

\filldraw [black] (2,5) circle (2pt);
\filldraw [black] (6,5) circle (2pt);

\filldraw [black] (2,3) circle (2pt);
\filldraw [black] (6,3) circle (2pt);

\draw (2.5,6)--(5.5,6)--(6.5,8)--(5.5,10)--(2.5,10)--(1.5,8)--(2.5,6);

\filldraw [black] (4,6) circle (2pt);
\filldraw [black] (5.5,6) circle (2pt);
\filldraw [black] (2.5,6) circle (2pt);
\filldraw [black] (5.5,10) circle (2pt);
\filldraw [black] (2.5,10) circle (2pt);
\filldraw [black] (4,10) circle (2pt);

\filldraw [black] (1.5,8) circle (2pt);
\filldraw [black] (6.5,8) circle (2pt);

\filldraw [black] (2,9) circle (2pt);
\filldraw [black] (6,9) circle (2pt);

\filldraw [black] (2,7) circle (2pt);
\filldraw [black] (6,7) circle (2pt);

\draw [dotted] (5.5,10)--(6,11);
\draw [dotted] (2,11)--(2.5,10);
\draw [dotted] (6.5,8)--(7.5,8);
\draw [dotted] (6.5,4)--(7.5,4);
\draw [dotted] (-2.5,6)--(-3.5,6);
\draw [dotted] (-2.5,2)--(-3.5,2);

\draw [dotted] (-1.5,8)--(-2,9);
\draw [dotted] (2,-1)--(1.5,0);
\draw [dotted] (6,1)--(5.5,2);

\draw [dotted] (-2,-1)--(-1.5,0);

\end{tikzpicture}}

\end{center}
\caption{Examples of infinite $(2,3)$-biregular graphs.}\label{fig:tree}
\end{figure}

The adjacency matrix of a given graph $G$, denoted by $A:=A(G)$, is an $n\times n$ symmetric matrix $(a_{ij})$ such that $a_{ij}=1$ if $i$ and $j$ are neighbors and $0$, otherwise. Then, we can write $\mathcal{N}(k)=\{\ell \in V: a_{k\ell}=1\},$ for any $k$. Note that the adjacency matrix for infinite graphs is well defined. A graph is said to be singular if its adjacency matrix $A$ is a singular matrix ($\det A =0$), otherwise the graph is said to be non-singular. All the graphs we consider are connected, i.e. for any $i,j\in V$ there exists a positive integer $n$ and a sequence of vertices $\gamma=(i_0,i_1,i_2,\ldots,i_n)$ such that $i_0=i$, $i_n=j$ and $(i_k,i_{k+1})\in E$ for all $k\in\{0,1,\ldots,n-1\}$. The sequence $\gamma$ is called a path connecting $i$ to $j$ with size $n$. The distance between two vertices $i$ and $j$, denoted  by $d(i,j)$, is the size, i.e. number of edges, in the shortest path connecting them. For simplicity, we consider only graphs which are simple, i.e. without multiple edges or loops. 

\smallskip


\subsection{The evolution algebras associated to a graph}

The evolution algebra induced by a graph $G$ is defined in \cite[Section 6.1]{tian} as follows.

\smallskip
\begin{defn}\label{def:eagraph}
Let $G=(V,E)$ a graph with adjacency matrix given by $A=(a_{ij})$. The evolution algebra associated to $G$ is the algebra $\A(G)$ with natural basis $S=\{e_i: i\in V\}$, and relations

\[
 \begin{array}{ll}\displaystyle
e_i \cdot e_i = \sum_{k\in V} a_{ik} e_k,&\text{for  }i \in  V,\\[.3cm]
\end{array}
\] 
\noindent
and $e_i \cdot e_j =0,\text{if }i\neq j.$
\end{defn}

Another way of stating the relation for $e_i \cdot e_i$, for $i\in V$, is to say $e_i^2 = \sum_{k\in \mathcal{N}(i)} e_k$.

\begin{exa}
Let $G$ be the $(2,3)$-biregular graph with $10$ vertices of Fig. \ref{fig:bipartite}. Then $\A(G)$ has natural basis $S=\{e_1,\ldots,e_{10}\}$, and relations

$$
\begin{array}{c}
\mathcal{A}(G): \left\{
\begin{array}{lllll}
 e_1^2= e_5 + e_8 +e_{10}, &  e_2^2=e_5  +e_6 + e_9,&e_3^2= e_6+ e_7+ e_9, & e_4^2=e_7 + e_8 +e_{10},\\[.2cm]
 e_5^2=e_1 + e_2, & e_6^2=e_2+e_3,& e_7^2 = e_3 + e_4, & e_8^2=e_1+ e_4,\\[.2cm]
 e_9^2=e_2+ e_3,& e_{10}^2=e_1+ e_4,  & e_i \cdot e_j =0, i\neq j . \\[.2cm]
\end{array}\right.
\end{array}
$$

\end{exa}

\smallskip
There is a second natural way to define an evolution algebra associated to $G=(V,E)$; it is the one induced by the symmetric random walk (SRW) on $G$. The SRW is a discrete time Markov chain $(X_n)_{n\geq 0}$ with state space given by $V$ and transition probabilities given by
$$\mathbb{P}(X_{n+1}=k|X_{n}=i)=\frac{a_{ik}}{\deg(i)},$$
where $i,k\in V$, $n\in \mathbb{N}$ and, as defined before, $\deg(i)=\sum_{k\in V} a_{ik}.$ Roughly speaking, the sequence of random variables $(X_n)_{n\geq 0}$ denotes the set of positions of a particle walking around the vertices of $G$; at each discrete-time step the next position is selected at random from the set of neighbors of the current one. Since the SRW is a discrete-time Markov chain we may define its related Markov evolution algebra.

\smallskip
\begin{defn}
Let $G=(V,E)$ be a graph with adjacency matrix given by $A=(a_{ij})$. We define the evolution algebra associated to the SRW on $G$ as the algebra $\A_{RW}(G)$ with natural basis $S=\{e_i: i\in V\}$, and relations given by
\[
\begin{array}{ll}\displaystyle
e_i \cdot e_i = \sum_{k\in V}\left( \frac{a_{ik}}{\deg(i)}\right)e_k,&\text{for }i  \in V,
\end{array}
\] 
\noindent
and $e_i \cdot e_j =0, \text{ if } i\neq j.$
\end{defn}

\begin{exa}
Consider again $G$ as being the $(2,3)$-biregular graph with $10$ vertices of Fig. \ref{fig:bipartite}. Then $\A_{RW}(G)$ has natural basis $S=\{e_1,\ldots,e_{10}\}$, and relations

$$
\begin{array}{c}
\mathcal{A}_{RW}(G): \left\{
\begin{array}{lllll}
 e_1^2=  \frac{1}{3}(e_5 + e_8 +e_{10}), &  e_2^2= \frac{1}{3}(e_5  +e_6 + e_9),&e_3^2=  \frac{1}{3}(e_6+ e_7+ e_9), & \\[.3cm]
e_4^2=  \frac{1}{3}(e_7 + e_8 +e_{10}),& e_5^2=\frac{1}{2} (e_1 + e_2), & e_6^2=\frac{1}{2}(e_2+e_3),\\[.3cm]
 e_7^2 =\frac{1}{2}( e_3 + e_4), & e_8^2= \frac{1}{2}(e_1+ e_4),& e_9^2= \frac{1}{2}(e_2+ e_3),& \\[.3cm]
e_{10}^2=\frac{1}{2}(e_1+ e_4),  & e_i \cdot e_j =0, i\neq j . \\[.3cm]

\end{array}\right.
\end{array}
$$

\end{exa}

\smallskip
The aim of this paper is to contribute with the discussion about the relation between the algebras $\A_{RW}(G)$ and  $\A(G)$ for a given graph $G$. We emphasize that this is one of the open problems stated by \cite{tian,tian2}, and which has been addressed partially by \cite{PMP}. Our approach will be the statement of conditions under which we can guarantee the existence or not of isomorphisms between these evolution algebras.


\section{Isomorphisms}

\subsection{Main results}

Before to address with the existence of isomorphisms between $\A_{RW}(G)$ and  $\A(G)$ for a given graph $G$, we start with a more general concept which is the isotopism of algebras introduced by Albert \cite{albert} as a generalization of that of isomorphism. This has been recently applied by \cite{Falcon/Falcon/Nunez/2017} to study two-dimensional evolution algebras.

\smallskip
\begin{defn}\label{def:isoto} \cite[Section 2.1]{Falcon/Falcon/Nunez/2017}
Let $\mathcal{A}$ and $\mathcal{B}$ be two evolution algebras over a field $\mathbb{K}$, and let $S=\{e_i: i\in V\}$ be a natural basis for $\mathcal{A}$. We say that a triple $(f,g,h)$, where $f,g,h$ are three non-singular $\mathbb{K}$-linear transformations 
from $\mathcal{A}$ into $\mathcal{B}$ is an isotopism if 
$$f(u)\cdot g(v) = h(u\cdot v),\;\;\; \text{ for all }u, v \in\mathcal{A}.$$
In this case we say that $\mathcal{A}$ and $\mathcal{B}$ are isotopic. In addition, the triple is called 
\begin{enumerate}[label=(\roman*)]
\item a strong isotopism if $f=g$ and we say that the algebras are strongly isotopic; 
\item an isomorphism if $f=g=h$ and we say that the algebras are isomorphic.
\end{enumerate}
\end{defn}

\smallskip
In the case of an isomorphism we write $f$ instead of $(f, f, f)$. To be isotopic, strongly isotopic or isomorphic are equivalence relations among algebras, and we denote these three relations, respectively, by $\sim$, $\simeq$ and $\cong$. The concept of isotopism allows a first formal connection to be found between $\A_{RW}(G)$ and  $\A(G)$.

\smallskip
\begin{theorem}
For any graph $G$, $\A(G)\simeq \A_{RW}(G)$. 
\end{theorem}

\begin{proof}
Consider two $\mathbb{K}$-linear maps, $f$ and $h$, from  $\A(G)$ to $\A_{RW}(G)$ defined by
$$f(e_i)=\sqrt{\deg(i)}\,e_i,\;\;\;\text{ and }\;\;\;h(e_i)=e_i,\;\;\;\text{ for all }i\in V.$$
Then, for $i\neq j$, $f(e_i)\cdot f(e_j)=\sqrt{\deg(i) \deg(j)} \left(e_i\cdot e_j \right)=0 = h(e_i\cdot e_j).$
On the other hand, for any $i\in V$, we have
$$f(e_i)\cdot f(e_i)=\deg(i) \, e_i^2 =\deg(i) \sum_{k\in V}\left( \frac{a_{ik}}{\deg(i)}\right)e_k =\sum_{k\in V} a_{ik} \, e_k,$$
while 
$$h(e_i^2)=h\left(\sum_{k\in V}  a_{ik}\,e_k\right)=\sum_{k\in V}  a_{ik} \,e_k,$$
and the proof is completed.
\end{proof}

Our next step is to obtain conditions on $G$ under which one have the existence or not of isomorphisms between $\A_{RW}(G)$ and $\A(G)$. This issue has been considered recently in \cite{PMP} for some well-known families of graphs. However, there is still a need for general results to address this question. The main result of the present work is a complete characterization of the problem for the case of {\color{black} finite} non-singular graphs.

\begin{theorem}\label{theo:criterio}
Let $G$ be a finite non-singular graph. $\A_{RW}(G)\cong \A(G)$ if, and only if, $G$ is a regular or a biregular graph. Moreover, if $\A_{RW}(G)\ncong \A(G)$ then the only homomorphism between them is the null map.
\end{theorem}

\begin{remark}
We are restricting our attention on the existence or not of algebra isomorphisms in the sense of Definition \ref{def:isoto}(ii). We empathize that our results can be easily adapted to deal with evolution homomorphisms or evolution isomorphisms. According to Tian, see \cite{tian}, the concept of evolution homomorphism is related to the one of homomorphism of algebras with an additional condition. More precisely, if $ \mathcal{A}$ and $\mathcal{B}$ are two evolution algebras over a field $\mathbb{K}$ and  $S=\{e_i: i\in V\}$  is a natural basis for $ \mathcal{A}$, then \cite[Definition 4]{tian} say that a linear transformation $g: \mathcal{A} \longrightarrow \mathcal{B}$ is an evolution homomorphism, if $g(a\cdot b)=g(a)\cdot g(b)$ for all $a,b \in  \mathcal{A} $  and  $\{g(e_i):i\in V\}$ can be complemented to a natural basis for $\mathcal{B}$. Furthermore, if an evolution homomorphism is one to one and onto, it is an evolution isomorphism. Using the terminology in \cite{Cabrera/Siles/Velasco} we can rewrite the definition of Tian by saying that an evolution homomorphism $g: \mathcal{A} \longrightarrow \mathcal{B}$ between evolution algebras $ \mathcal{A}$ and $\mathcal{B}$ is an  homomorphism such that the evolution algebra $\Im(f)$ has the extension property.  
\end{remark}

The proof of Theorem \ref{theo:criterio} rely on a mix of results which holds for general graphs, meaning not necessarily {\color{black}finite} and non-singular graphs, together with a description of the isomorphisms for the case of {\color{black} finite }non-singular graphs. For the sake of clarity we left the proof for the next section. In what follows we discuss some examples.

\begin{exa}Friendship graph $F_n$. Let us consider the friendship graph $F_n$, which is a finite graph with $2n+1$ vertices and $3n$ edges constructed by joining $n$ copies of the triangle graph with a common vertex (see Figure \ref{fig:friendhsipproof}). We shall see that $\rank (A)= n$, which implies by  Theorem \ref{theo:criterio}, because the graph is neither regular nor biregular,  that if $f:\mathcal{A}(G)\longrightarrow \mathcal{A}_{RW}(G)$ is an homomorphism, then $f$ is the null map. This results has been stated in \cite[Proposition 3.4]{PMP}, and therefore it is a Corollary of our Theorem  \ref{theo:criterio}.

We assume the vertices of $F_n$ labelled as in Figure \ref{fig:friendhsipproof}, with the central vertex labelled by $2n+1$.

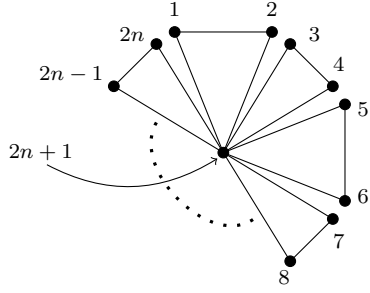
\begin{figure}[h!]\label{fig:friendhsipproof}
\begin{center}
\begin{tikzpicture}[scale=0.8]

\draw (0,0) -- (2,0.8)--(2,-0.8)--(0,0);
\draw (0,0) -- (-0.8,2)--(0.8,2)--(0,0);
\draw (0,0) -- (-1.8,1.1)--(-1.1,1.8)--(0,0);
\draw (0,0) -- (1.8,1.1)--(1.1,1.8)--(0,0);
\draw (0,0) -- (1.8,-1.1)--(1.1,-1.8)--(0,0);

\filldraw [black] (0,0) circle (2.5pt);
\draw (-1.5,1.7) node[above,font=\footnotesize] {$2n$};
\filldraw [black] (2,0.8) circle (2.5pt);
\draw (1.5,1.7) node[above,font=\footnotesize] {$3$};
\filldraw [black] (2,-0.8) circle (2.5pt);
\draw (1,-1.8) node[below,font=\footnotesize] {$8$};
\filldraw [black] (-0.8,2) circle (2.5pt);
\draw (-0.8,2.1) node[above,font=\footnotesize] {$1$};
\filldraw [black] (0.8,2) circle (2.5pt);
\draw (0.8,2.1) node[above,font=\footnotesize] {$2$};

\filldraw [black] (1.1,-1.8) circle (2.5pt);
\filldraw [black] (1.1,1.8) circle (2.5pt);
\filldraw [black] (-1.1,1.8) circle (2.5pt);

\filldraw [black] (-1.8,1.1) circle (2.5pt);
\draw (-2.5,1.3) node[font=\footnotesize] {$2n-1$};
\filldraw [black] (1.8,1.1) circle (2.5pt);
\draw (1.9,1.2) node[above,font=\footnotesize] {$4$};
\filldraw [black] (1.8,-1.1) circle (2.5pt);
\draw (2.3,1) node[below,font=\footnotesize] {$5$};
\draw (2.3,-1) node[above,font=\footnotesize] {$6$};
\draw (1.9,-1.2) node[below,font=\footnotesize] {$7$};

\draw (-3,0) node[font=\footnotesize] {$2n+1$};
\path [->, bend right]  (-2.9,-0.2) edge (-0.1,-0.1);

\draw [very thick,loosely dotted]    (-1.1,0.5) to[out=240,in=-145] (0.5,-1.1);

\end{tikzpicture}

\caption{Friendship graph.}
\end{center}
\end{figure}

Then the adjacency matrix $A$ of the graph has elements
$$a_{ij}=\left\{
\begin{array}{cl}
1,& \text{ if }i\text{ is odd (even) and }j=i+1\; (j=i-1),\\[.2cm] 
1,& \text{ if } i=2n+1\text{ and }j\in\{1,\ldots, n\}\\[.2cm]
0,& \text{ other case.}
\end{array}\right.
$$

That is, $A$ is given by

$$
\begin{bmatrix}
0&1&0&0&\cdots&0&1\\[.2cm]
1&0&0&0&\cdots &0&1\\[.2cm]
0&0&0&1&\cdots &0&1\\[.2cm]
\vdots & \vdots& \vdots&\vdots &\ddots & \vdots&\vdots\\[.2cm]
0&0&0&0&\cdots &0&1\\[.2cm]
1&1&1&1&\cdots &1&0
\end{bmatrix}.
$$

Denote by $C_i$ the $ith$-column of matrix $A$, for $i\in\{1,\ldots,2n+1\}$, and assume that  $\sum_{i=1}^{2n+1} \alpha_i C_i = 0,$ where $\alpha_i$ is a constant, for  $i\in\{1,\ldots,2n+1\}$. Now it is not difficult to see that the following equations hold:
\begin{eqnarray}
\alpha_k + \alpha_{2n+1} = 0,& \text{ for }k\in\{1,\ldots,2n\},\label{eq:frien1}\\[.2cm]
\sum_{i=1}^{2n}\alpha_i =0.\label{eq:frien2}&
\end{eqnarray}
Thus, by adding \eqref{eq:frien1} and \eqref{eq:frien2} we obtain $\sum_{i=1}^{2n}\alpha_i + 2n (\alpha_{2n+1})=0$ which, together with \eqref{eq:frien2}, implies $\alpha_{2n+1}=0$. Thus we can conclude, now from \eqref{eq:frien1}, that $\alpha_k =0$ also for $k\in\{1,\ldots,2n\}$. This in turns implies that $\{C_1,\ldots,C_{2n+1}\}$ forms a linearly independent set of vectors and hence $\rank(A)=n$.

\end{exa}

\smallskip
A natural question that needs to be raised is if the result stated in Theorem \ref{theo:criterio} holds for {\color{black}finite} singular graphs also. In the sequel we provide some examples suggesting a positive answer. 

\smallskip
\begin{exa}\label{exa:regular}
Consider the $3$-regular graph $G$ represented as in Fig. 2.2. 
  
\begin{figure}[!h]
\begin{center}
\begin{tikzpicture}[scale=0.7]

\draw (-1,1) -- (-3,1) -- (-3,-1)--(-1,-1);
\draw (3,1) -- (5,1) -- (5,-1)--(3,-1)--(5,1);
\draw (3,1) -- (5,-1);
\draw (-3,1) -- (-1,-1);
\draw (-3,-1) -- (-1,1);
\draw (2,0) --(3,-1);
\draw (2,0) --(3,1);
\draw (0,0) -- (-1,1);
\draw (0,0) -- (-1,-1);
\draw (0,0) -- (2,0);


\filldraw [black] (0,0) circle (2.5pt); \draw (0.3,-0.3) node {$1$};
\filldraw [black] (-1,1) circle (2.5pt); \draw (-1,1.4) node {$2$};
\filldraw [black] (-3,1) circle (2.5pt); \draw (-3,1.4) node {$3$};
\filldraw [black] (-3,-1) circle (2.5pt); \draw (-3,-1.4) node {$4$};
\filldraw [black] (-1,-1) circle (2.5pt); \draw (-1,-1.4) node {$5$};
\filldraw [black] (2,0) circle (2.5pt); \draw (1.7,-0.3) node {$6$};
\filldraw [black] (3,1) circle (2.5pt); \draw (3,1.4) node {$7$};
\filldraw [black] (5,1) circle (2.5pt); \draw (5,1.4) node {$8$};
\filldraw [black] (5,-1) circle (2.5pt); \draw (5,-1.4) node {$9$};
\filldraw [black] (3,-1) circle (2.5pt); \draw (3,-1.4) node {$10$};

\end{tikzpicture}
\caption{$3$-regular singular graph.}
\label{FIG:3reg}
\end{center}
\end{figure}
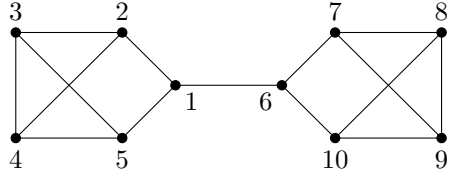

\noindent
The evolution algebras induced by $G$, and by the random walk on $G$, respectively, have natural basis $\{e_1,e_2,e_3,e_4,e_5\}$ and relations given by:

\smallskip
$$
\begin{array}{cc}
\mathcal{A}(G): \left\{
\begin{array}{l}
 e_1^2= e_2 + e_5 + e_6,\\[.2cm]
 e_2^2=e_5^2= e_1 + e_3 + e_4,\\[.2cm]
 e_3^2=e_2 + e_4+ e_5,\\[.2cm]
 e_4^2=e_2 + e_3+ e_5,\\[.2cm]
 e_6^2=e_1+e_7+e_{10},\\[.2cm]
e_7^2 = e_{10}^2=e_6 + e_8 + e_9,\\[.2cm]
e_8^2=e_7 + e_9 + e_{10},\\[.2cm]
 e_i \cdot e_j =0, i\neq j,
\end{array}\right.
&
\mathcal{A}_{RW}(G): \left\{
\begin{array}{l}
 e_1^2= \frac{1}{3}\,e_2 +\frac{1}{3}\, e_5 +\frac{1}{3}\, e_6,\\[.2cm]
 e_2^2=e_5^2=\frac{1}{3}\, e_1 +\frac{1}{3}\, e_3 +\frac{1}{3}\, e_4,\\[.2cm]
 e_3^2=\frac{1}{3}\,e_2 + \frac{1}{3}\,e_4+ \frac{1}{3}\,e_5,\\[.2cm]
 e_4^2=\frac{1}{3}\,e_2 +\frac{1}{3}\, e_3+\frac{1}{3}\, e_5,\\[.2cm]
 e_6^2=\frac{1}{3}\,e_1+\frac{1}{3}\,e_7+\frac{1}{3}\,e_{10},\\[.2cm]
e_7^2 = e_{10}^2=\frac{1}{3}\,e_6 + \frac{1}{3}\,e_8 +\frac{1}{3}\, e_9,\\[.2cm]
e_8^2=\frac{1}{3}\,e_7 + \frac{1}{3}\,e_9 + \frac{1}{3}\,e_{10},\\[.2cm]
 e_i \cdot e_j =0, i\neq j.
\end{array}\right.
\end{array}
$$

\vspace{.5cm}
\noindent
Note that $\mathcal{N}(2)=\mathcal{N}(5)=\{1,3,4\}$, and $\mathcal{N}(7)=\mathcal{N}(10)=\{6,8,9\}$, implies $\det A =0$. Moreover, as $G$ is a $3$-regular graph, we have by \cite[Theorem 3.2(i)]{PMP} that $\mathcal{A}_{RW}(G) \cong \mathcal{A}(G)$. It is not difficult to see that the map $f:\mathcal{A}_{RW}(G)\longrightarrow \mathcal{A}(G)$ defined by $f(e_i) = (1/3)\, e_i$, for any $i\in V$ is an isomorphism.

\end{exa}

\smallskip
\begin{exa}\label{exa:bipartite}
Consider the complete bipartite graph $K_{m,n}$ with partitions of sizes $m\geq 1$ and $n\geq 1$, respectively, and assume that the set of vertices is partitioned into the two subsets $V_1:=\{1,\ldots,m\}$ and $V_2:=\{m+1,\ldots,m+n\}$ (see Fig. 2.3). 

\begin{figure}[!h]
\label{FIG:bipartite}
\begin{center}
\begin{tikzpicture}[scale=0.7]

\draw (0,0) -- (4,-1);
\draw (0,0) -- (4,-4);
\draw (0,0) -- (4,-2.5);

\draw (0,-1) -- (4,-1);
\draw (0,-1) -- (4,-4);
\draw (0,-1) -- (4,-2.5);

\draw (0,-2) -- (4,-1);
\draw (0,-2) -- (4,-4);
\draw (0,-2) -- (4,-2.5);

\draw (0,-3) -- (4,-1);
\draw (0,-3) -- (4,-4);
\draw (0,-3) -- (4,-2.5);

\draw (0,-4) -- (4,-1);
\draw (0,-4) -- (4,-4);
\draw (0,-4) -- (4,-2.5);

\draw (0,-5) -- (4,-1);
\draw (0,-5) -- (4,-4);
\draw (0,-5) -- (4,-2.5);


\filldraw [black] (0,0) circle (2.5pt); \node at (-0.5,0) {$1$};
\filldraw [black] (0,-1) circle (2.5pt); \node at (-0.5,-1) {$2$};
\filldraw [black] (0,-2) circle (2.5pt); \node at (-0.5,-2) {$3$};
\filldraw [black] (0,-3) circle (2.5pt); \node at (-0.5,-3) {$4$};
\filldraw [black] (0,-4) circle (2.5pt); \node at (-0.5,-4) {$5$};
\filldraw [black] (0,-5) circle (2.5pt); \node at (-0.5,-5) {$6$};

\filldraw [black] (4,-2.5) circle (2.5pt); \node at (4.5,-2.5) {$8$};
\filldraw [black] (4,-4) circle (2.5pt); \node at (4.5,-4) {$9$};
\filldraw [black] (4,-1) circle (2.5pt); \node at (4.5,-1) {$7$};

\end{tikzpicture}
\caption{Complete bipartite graph $K_{6,3}$}
\end{center}
\end{figure}
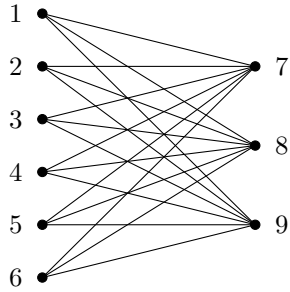

\noindent
It is not difficult to see that $\det A =0$. The associated evolution algebras $\mathcal{A}_{RW}(K_{m,n})$ and $\mathcal{A}(K_{m,n})$ are defined in \cite{PMP}. Indeed, by \cite[Theorem 3.2(ii)]{PMP} we have that  $\mathcal{A}_{RW}(K_{m,n}) \cong \mathcal{A}(K_{m,n})$. Moreover, let $f_{\pi}:\mathcal{A}_{RW}(K_{m,n})\longrightarrow \mathcal{A}(K_{m,n})$ be defined by
 $$
f_{\pi}(e_i)=\left\{
\begin{array}{ll}
m^{-1/3}n^{-2/3} e_{\pi(i)}, &\text{for }i\in V_1;\\[.2cm]
m^{-2/3}n^{-1/3} e_\pi(i), &\text{for }i\in V_2, 
\end{array}\right.
$$
where $\pi \in S_{m+n}$ is such that $\pi(i)\in V_1$ if, and only if, $i\in V_1$. 
\end{exa}

\begin{exa}\label{exa:tree}
A tree $T$ with $m+n+2$ vertices and diameter $3$ may be represented as in Fig. 2.4. The set of vertices may be partitioned in $V_1:=\{1,\ldots,m\}$, $V_2:=\{m+1,\ldots,m+n\}$, and $\{m+n+1,m+n+2\}$, in such a way $\mathcal{N}(i) =\{ n+m+1\}$ for any $i\in V_1$, $\mathcal{N}(i) = \{n+m+2\}$ for any $i\in V_2$, and $n+m+1$ and $n+m+2$ are neighbors. Then $\det A =0$. 

\begin{figure}[!h]
\label{FIG:tree}
\begin{center}
\begin{tikzpicture}[scale=0.7]

\draw (2,0) -- (4,1);

\draw (2,0) --(4,-2);
\draw (2,0) --(4,2);
\draw (2,0) --(4,0);
\draw (0,0) -- (-2,1);
\draw (0,0) -- (-2,2);
\draw (0,0) -- (-2,0);
\draw (0,0) -- (-2,-2);
\draw (0,0) -- (2,0);


\filldraw [black] (0,0) circle (2.5pt); \draw (0,-0.4) node {$u$};
\filldraw [black] (-2,2) circle (2.5pt); \draw (-2.5,2) node {$1$};
\filldraw [black] (-2,0) circle (2.5pt); \draw (-2.5,0) node {$3$};
\draw (-2,-0.8) node {$\vdots$};
\draw (4,-0.8) node {$\vdots$};
\filldraw [black] (2,0) circle (2.5pt); \draw (2,-0.4) node {$v$};
\filldraw [black] (4,2) circle (2.5pt); \draw (5,2) node {$m+1$};
\filldraw [black] (4,0) circle (2.5pt); \draw (5,0) node {$m+3$};
\filldraw [black] (4,1) circle (2.5pt); \draw (5,1) node {$m+2$};
\filldraw [black] (4,-2) circle (2.5pt); \draw (5,-2) node {$m+n$};
\filldraw [black] (-2,-2) circle (2.5pt); \draw (-2.5,-2) node {$m$};
\filldraw [black] (-2,1) circle (2.5pt); \draw (-2.5,1) node {$2$};
\filldraw [white] (-1,2) circle (2.5pt);
\filldraw [white] (-1,-2) circle (2.5pt);
\filldraw [white] (1,2) circle (2.5pt);
\filldraw [white] (1,-2) circle (2.5pt);

\end{tikzpicture}
\caption{Tree with $m+n+2$ vertices and diameter $3$. Here $u:=m+n+1$ and $v:=m+n+2$.}
\end{center}
\end{figure}
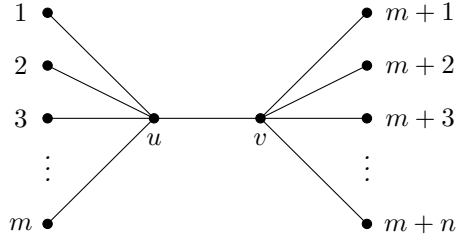

We shall see that this is an example of graph where the only homomorphism is the null map. For the sake of simplicity we consider the case $m=n=2$; the general case could be checked following the same arguments as below with some additional work. For $m=n=2$, the tree $T$ induces the following evolution algebras: take the natural basis $\{e_1,e_2,e_3,e_4,e_5,e_6\}$ and the relations given by

\smallskip
$$
\begin{array}{cc}
\mathcal{A}(T): \left\{
\begin{array}{l}
 e_1^2 = e_2^2 = e_5,\\[.2cm]
  e_3^2 = e_4^2 = e_6,\\[.2cm]
 e_5^2=e_1 + e_{2}+e_6,\\[.2cm]
  e_6^2=e_3 + e_{4}+e_5,\\[.2cm]
 e_i \cdot e_j =0,  i\neq j,
\end{array}\right.
&
\mathcal{A}_{RW}(T): \left\{
\begin{array}{l}
 e_1^2 = e_2^2 = e_5,\\[.2cm]
  e_3^2 = e_4^2 = e_6,\\[.2cm]
 e_5^2=\frac{1}{3}\,e_1 +\frac{1}{3}\, e_{2}+\frac{1}{3}\,e_6,\\[.2cm]
  e_6^2=\frac{1}{3}\,e_3 +\frac{1}{3}\, e_{4}+\frac{1}{3}\, e_5,\\[.2cm]
 e_i \cdot e_j =0,  i\neq j.
\end{array}\right.
\end{array}
$$

Assume $f:\mathcal{A}_{RW}(T) \longrightarrow \mathcal{A}(T)$ is an homomorphism such that for any $i\in\{1,2,3,4,5,6\}$ 
$$f(e_i )= \sum _{k=1} ^{6}t_{ik}e_k,$$
where the $t_{ik}$'s are scalars. Thus, 
\begin{eqnarray}
f(e_1^2) = f(e_2^2) = f(e_5) &=& \sum_{k=1}^{6} t_{5k} e_k,\label{eq:pri12}\\
f(e_3^2) = f(e_4^2) = f(e_6) &=& \sum_{k=1}^{6} t_{6k} e_k,\label{eq:pri34}\\
f(e_5^2) = f\left(\frac{1}{3}(e_1 + e_2 + e_6)\right)  &=& \frac{1}{3}\sum_{k=1}^{6}(t_{1k}+t_{2k}+t_{6k})e_k,\label{eq:pri5}\\
f(e_6^2) = f\left(\frac{1}{3}(e_3 + e_4 + e_5)\right)  &=& \frac{1}{3}\sum_{k=1}^{6}(t_{3k}+t_{4k}+t_{5k})e_k,\label{eq:pri6}\\
f(e_i)\cdot f(e_j) &=&\sum_{k=1}^{6}\left(\sum_{\ell \in \mathcal{N}(k)}t_{i\ell}t_{j\ell}\right) e_k,\label{eq:priij}
\end{eqnarray}
for any $i, j\in V$, which together with 
\begin{equation}\label{eq:isopro}
f(e_i \cdot e_j)=f(e_i)\cdot f(e_j), \text{ for any }i,j\in V,
\end{equation}
imply the following set of equations. If $i\neq j$, then $0=f(e_i\cdot e_j)$ and we obtain by \eqref{eq:priij} and \eqref{eq:isopro}:

\begin{eqnarray}
t_{ik}t_{jk} &=&0, \text{ for }k\in\{5,6\},\label{eq:exT1}\\
t_{ik}t_{jk} + t_{i\,k+1}t_{j\,k+1}&=&0, \text{ for }k\in\{1,3\}.\label{eq:exT2}
\end{eqnarray}

By \eqref{eq:isopro} with $i=j$ and $i\in\{1,2,3,4\}$ we obtain by \eqref{eq:pri12} and \eqref{eq:pri34} the following: if $k=5$ and $\ell \in\{1,2\}$, or if $k=6$ and $\ell\in \{3,4\}$, it holds

\begin{eqnarray}
t_{k5} &=&t_{\ell \,1}^2 + t_{\ell \,2}^2 + t_{\ell \,6}^2,\label{eq:exT3}\\
t_{k6}&=&t_{\ell \,3}^2 + t_{\ell \,4}^2 + t_{\ell \,5}^2,\label{eq:exT4}\\
t_{k1}=t_{k2}&=&t_{\ell 5}^2,\label{eq:exT5}\\
t_{k3}=t_{k4}&=&t_{\ell 6}^2.\label{eq:exT6}
\end{eqnarray}

On the other hand, by   \eqref{eq:priij}  and   \eqref{eq:isopro} with $i=j=5$ and \eqref{eq:pri5} we obtain: if $k=5$ and $\ell \in\{1,2\}$, or if $k=6$ and $\ell\in \{3,4\}$, we get 

\begin{eqnarray}
3\, t_{5k}^2 &=& t_{1 \ell} + t_{2 \ell} + t_{6 \ell},\label{eq:exT7}\\
3(t_{51}^2 + t_{52}^2 + t_{56}^2) &=& t_{15} + t_{25} + t_{65},\label{eq:exT8}\\
3(t_{53}^2 + t_{54}^2 + t_{55}^2) &=& t_{16} + t_{26} + t_{66}.\label{eq:exT9}
\end{eqnarray}

Finally, following \eqref{eq:pri6},  \eqref{eq:priij}   and \eqref{eq:isopro} with $i=j=6$: if $k=5$ and $\ell \in\{1,2\}$, or if $k=6$ and $\ell\in \{3,4\}$

\begin{eqnarray}
3\, t_{6k}^2 &=& t_{3 \ell} + t_{4 \ell} + t_{5 \ell},\label{eq:exT10}\\
3(t_{61}^2 + t_{62}^2 + t_{66}^2) &=& t_{35} + t_{45} + t_{55},\label{eq:exT11}\\
3(t_{63}^2 + t_{64}^2 + t_{65}^2) &=& t_{36} + t_{46} + t_{56}.\label{eq:exT12}
\end{eqnarray}

By \eqref{eq:exT1}, for $k\in\{5,6\}$, we have $t_{ik} =0$ for all $ i\in V$, or there exists at most one $i\in V$ such that $t_{ik}  \not = 0$ and  $t_{jk} =0$ for all $j  \not = i$. This implies, by \eqref{eq:exT5} and \eqref{eq:exT6} that 
\begin{eqnarray}\label{eq:exT13}
t_{ki}=t_{i k}=0,& \text{ for }k\in\{5,6\}\text{ and }i\in \{1,2,3,4\}. 
\end{eqnarray}

\noindent
From now on we shall consider two different cases, namely, $t_{55} = t_{65}=0$ or $t_{55} = 0$ and $t_{65}\neq0$; indeed it should be three cases but the case $t_{55} \neq 0$ and  $t_{65}= 0$ is analogous to the last one. Note that we already have, see \eqref{eq:exT13}, $t_{i5}=0$ for $i\in\{1,2,3,4\}$.\\

\noindent
{\bf Case $1$:} $t_{55} = t_{65}=0$. In this case we get by \eqref{eq:exT3}, that 
\begin{eqnarray}\label{eq:exT14}
t_{i1}=t_{i2}=0,\text{ for }i\in\{1,2,3,4\}.
\end{eqnarray}
In addition, by \eqref{eq:exT8} we have $t_{56}=0$, and this in turns implies by \eqref{eq:exT4}, for $k=5$, 
\begin{eqnarray}\label{eq:exT15}
t_{i3}=t_{i4}=0,&\text{ for }i\in \{1,2\}. 
\end{eqnarray}
Analogously,  \eqref{eq:exT11} implies $t_{66}=0$, which in turns implies by \eqref{eq:exT4}, for $k=6$,
\begin{eqnarray}\label{eq:exT16}
t_{i3}=t_{i4}=0,& \text{ for }i\in\{3,4\}.
\end{eqnarray} 
Therefore, as $t_{i5}=0$ for any $i\in\{1,2,3,4,5,6\}$, $t_{56}=t_{66}=0$, and \eqref{eq:exT13}-\eqref{eq:exT16} hold, we conclude that $f$ is the null map.\\

\noindent
{\bf Case $2$:} $t_{55} = 0$ and  $t_{65}\neq 0$. As before, $t_{55} = 0$ implies by \eqref{eq:exT3}, for $k=5$
\begin{eqnarray}\label{eq:exT17}
t_{i1}=t_{i2}=0,&\text{ for }i\in\{1,2\},
\end{eqnarray}
and by \eqref{eq:exT11} together with \eqref{eq:exT13} we have $t_{66}=0$, which implies \eqref{eq:exT16}. Now, observe that it should be $t_{56}\neq0$; othercase \eqref{eq:exT12} and \eqref{eq:exT13} lead us to $t_{65}=0$, which is a contradiction. So assume $t_{56}\neq0$. By \eqref{eq:exT4}, for $k=5$ we have 
$$2 t_{56} = t_{13}^2+t_{14}^2+t_{23}^2+t_{24}^2 = t_{13}^2+2t_{13}t_{23}+t_{23}^2+t_{14}^2+2t_{14}t_{24}+t_{24}^2,$$
where the last equality comes from \eqref{eq:exT2} for $k=3$. In other words, we have
$$2 t_{56} = (t_{13}+t_{23})^2+(t_{14}+ t_{24})^2,$$
which by \eqref{eq:exT7} for $k=6$, and using that $t_{63}=t_{64}=0$ by \eqref{eq:exT13}, leads us to $2 t_{56} = 9  t_{56}^4$. Then $t_{56}=\left(2/9\right)^{1/3}$. Now \eqref{eq:exT8} and \eqref{eq:exT13} imply $3 t_{56}^2 = t_{65}$, and then $t_{65}=(4/3)^{1/3}\approx 1.1$. On the other hand, we could discover the value of $t_{65}$ following the same steps as the ones for $t_{56}$. In that direction one get by \eqref{eq:exT12} and \eqref{eq:exT13} that $t_{65}=(2/243)^{1/6}\approx 0.45$, which is a contradiction.



Our analysis of Case $2$ lead us to conclude that the only option is the one of Case $1$. Therefore, $f$ must  be the null map.

\end{exa}

Examples \ref{exa:regular}, \ref{exa:bipartite} and \ref{exa:tree} consider different singular graphs. From different arguments and applying previous results we have checked that either there exists an isomorphism between $\mathcal{A}_{RW}(G)$ and $\mathcal{A}(G)$, or the only homomorphism between these algebras is the null map. This leads us to think that Theorem \ref{theo:criterio} holds for {\color{black}finite} singular graphs also. However, further work needs to be carried out to establish whether this is true or not so we state it as a conjecture for future research.

\begin{conjecture}\label{conjecture}
Let $G$ be a {\color{black}finite} graph. $\A_{RW}(G)\cong \A(G)$ if, and only if, $G$ is a regular or a biregular graph. Moreover, if $\A_{RW}(G)\ncong \A(G)$ then the only homomorphism between them is the null map.
\end{conjecture}




\subsection{Some results for general graphs}

As stated in the previous Section, the existence of isomorphisms between $\mathcal{A}(G)$ and $\mathcal{A}_{RW}(G)$ has been stablished in \cite{PMP} for the particular case of regular and complete bipartite graphs. As we show next this result can be extended for biregular graphs.

\begin{prop}\label{theo:generalization}
Let $G=(V_1,V_2,E)$ be a biregular graph. Then $\A(G) \cong \A_{RW}(G)$.
\end{prop}

\begin{proof}
Assume that $G=(V_1,V_2,E)$  is a $(d_1,d_2)$-biregular  graph and consider the linear map $f:\mathcal{A}(G)\longrightarrow \mathcal{A}_{RW}(G)$ defined by 
\begin{equation}\label{eq:iso}
f(e_i)=\left\{
\begin{array}{cl}
\left(d_1^2 d_2\right)^{1/3}\, e_i,& \text{ if }i\in V_1,\\[.2cm]
\left(d_1 d_2^2\right)^{1/3}\, e_i,& \text{ if }i\in V_2.
\end{array}\right.
\end{equation}
Thus defined $f$ is an isomorphism between $\A(G)$ and $\A_{RW}(G)$. 
\end{proof}

By \cite[Theorem 3.2]{PMP} and Proposition \ref{theo:generalization} we have that $\mathcal{A}_{RW}(G) \cong \mathcal{A}(G)$ provided $G$ is either a regular or a biregular graph. At this point, the reader could ask if the converse is true. The following result sheds some light on this question.

\begin{prop}\label{theo:sufficient}
Let $G=(V,E)$ be a graph. Assume that there exist an isomorphism $f:\mathcal{A}(G)\longrightarrow \mathcal{A}_{RW}(G)$ defined by \begin{equation}\label{eq:isofunc}
f(e_i) = \alpha_i e_{\pi(i)},\;\;\;\text{ for all }i\in V,
\end{equation}
where $\alpha_i \neq 0$ is a scalar, for $i\in V$, and $\pi$ is an element of the  symmetric group {\color{black}$S_{V}$}. Then $G$ is a biregular graph  or a regular graph. 
\end{prop}

\begin{proof}
Assume that there map $f:\mathcal{A}(G)\longrightarrow \mathcal{A}_{RW}(G)$ is an isomorphism defined by $f(e_i) = \alpha_i e_{\pi(i)}$, where $\alpha_i \neq 0$, for $i\in V$, and $\pi \in {\color{black}S_{V}}$.   Since $f$ is linear  we have that 
\begin{equation}
 f(e_{i}^2)= f \left( \sum_{\ell \in \mathcal{N}(i) }e_{\ell} \right) =  \sum_{\ell \in \mathcal{N}(i) } f(e_{\ell})  =  \sum_{\ell \in \mathcal{N}(i) } \alpha_{\ell} e_{\pi(\ell)} 
\end{equation}
for $i\in V$. On the other hand, since  $f$ is an homomorphism then we have for any $i\in V$:
\begin{equation}
f(e_{i}^2)=f(e_i)\cdot f(e_i) = \alpha_ i^2 e_{\pi(i)}^2= \alpha_ i^2 \sum_{\ell \in \mathcal{N}(\pi(i))} \frac{1}{deg(\pi(i))} \, e_{\ell} = \frac{\alpha_ i^2}{\deg(\pi(i))} \sum_{\ell \in \mathcal{N}(\pi(i))} \,e_{\ell}.
\end{equation}
Then 
\begin{equation} 
\sum_{\ell \in \mathcal{N}(i) } \alpha_{\ell} e_{\pi(\ell)} = \frac{\alpha_ i^2}{\deg(\pi(i))} \sum_{\ell \in \mathcal{N}(\pi(i))} \,e_{\ell}, \text{  for all } i \in V.
\end{equation}
Therefore  $  \mathcal{N}(i)=  \mathcal{N}(\pi(i))$ i. e. $ \deg(i)=\deg(\pi(i))$ for all $i\in V$. It follows
\begin{equation} \label{eq:prova2a}
\alpha_\ell = \frac{\alpha_{i}^2}{\deg(i)}, \hspace{0.3cm} \text { for all }\, \ell \in \mathcal{N}(i)
\end{equation}
This implies that,  for  
\begin{equation} \label{eq:prova2b}
\ell_1, \ell_2 \in \mathcal{N}(i), \,\,\,\,\alpha_{\ell_1}=\alpha_{\ell_2}.
\end{equation}

Given  $i\in V$,  if  $\ell \in \mathcal{N}(i)$ then  $i \in \mathcal{N}(\ell)$,  hence  by \eqref{eq:prova2a}
\begin{equation} \label{eq:prova2c}
\alpha_i = \frac{\alpha_{\ell}^2}{\deg(\ell)} \hspace{0.3cm} \text{ for all  }\ell \in \mathcal{N}(i)
\end{equation}
So by \eqref{eq:prova2b} and \eqref{eq:prova2c} for $\ell_1, \ell_2 \in \mathcal{N}(i)$
\begin{equation}
\frac{\alpha_{\ell_{1}}^2}{\deg(\ell_{1})} = \alpha_i = \frac{\alpha_{\ell_{2}}^2}{\deg(\ell_{2})} . \end{equation}
As a consequence, we obtain the following condition on the degrees in the graph:

\begin{equation}\label{eq:prova2}
\text{ for any }i\in V, \text{ if  }\ell_1, \ell_2 \in \mathcal{N}(i) \text{ then }\deg(\ell_1)=\deg(\ell_2).
\end{equation}
Now let us fix a vertex, say $1$, and note that by \eqref{eq:prova2} we have $\deg(\ell)=\deg(1)$ for any $\ell \in V$ such that there is a path of even size from $1$ to $\ell$ (see Fig. \ref{fig:vertices}).

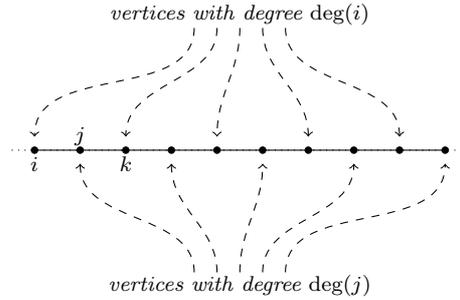
\begin{figure}[!h]

\begin{center}

\begin{tikzpicture}[scale=0.6]

\draw (-4,0) -- (5,0);
\draw [dotted] (-4.5,0) -- (5.5,0);

\filldraw [black] (-4,0) circle (2pt);
\draw (-4,-0.3) node[font=\footnotesize] {$i$};
\filldraw [black] (-3,0) circle (2pt);
\draw (-3,0.3) node[font=\footnotesize] {$j$};
\filldraw [black] (-2,0) circle (2pt);
\draw (-2,-0.3) node[font=\footnotesize] {$k$};
\filldraw [black] (-1,0) circle (2pt);
\filldraw [black] (0,0) circle (2pt);
\filldraw [black] (5,0) circle (2pt);
\filldraw [black] (4,0) circle (2pt);
\filldraw [black] (3,0) circle (2pt);
\filldraw [black] (2,0) circle (2pt);
\filldraw [black] (1,0) circle (2pt);

\draw [dashed,->] (-0.5,2.7) to [out=270,in=90] (-4,0.3);
\draw [dashed,->] (0,2.7) to [out=270,in=90] (-2,0.3);
\draw [dashed,->] (0.5,2.7) to [out=270,in=90] (0,0.3);
\draw [dashed,->] (1,2.7) to [out=270,in=90] (2,0.3);
\draw [dashed,->] (1.5,2.7) to [out=270,in=90] (4,0.3);

\draw (0.5,3) node[font=\footnotesize] {{\it vertices with degree} $\deg(i)$};
\draw (0.5,-3) node[font=\footnotesize] {{\it vertices with degree} $\deg(j)$};

\draw [dashed,->] (-0.5,-2.7) to [out=90,in=270] (-3,-0.3);
\draw [dashed,->] (0,-2.7) to [out=90,in=270] (-1,-0.3);
\draw [dashed,->] (0.5,-2.7) to [out=90,in=270] (1,-0.3);
\draw [dashed,->] (1,-2.7) to [out=90,in=270] (3,-0.3);
\draw [dashed,->] (1.5,-2.7) to [out=90,in=270] (5,-0.3);

\end{tikzpicture}

\caption{Ilustration of the application of condition \eqref{eq:prova2} to a given path of the graph. Since $i\in \N(j)$ and $j\in \N(k)$ it should be $\deg(k) = \deg(i)$. This argument may be extended to the whole graph showing that at the end there are at most two different degrees in the vertices of the graph.}\label{fig:vertices}
\end{center}
\end{figure}

Analogously, we have $\deg(\ell_1)=\deg(\ell_2)$ for any $\ell_1,\ell_2 \in V$ such that there is a path of odd size from $1$ to $\ell_k$, $k\in\{1,2\}$. Now let us define $V_1:=\{j \in V: d(1,j) \text{ is } even \}$ and $V_2:=\{j \in V: d(1,j) \text{ is } odd\}$. Notice that by our previous comments our definition of $V_1$ and $V_2$ is enough to guarantee that $\deg(i)=\deg(j)$ for $i,j\in V_k$, and $k\in\{1,2\}$. If every edge  on $G$ connects a vertex in $V_1$ to one in $V_2$, then $G$ is a  biregular graph.  In the opposite case, if there exist $i,j \in V_1$ such that $i \in \N(j)$, we claim that $G$ is a $\deg(1)$-regular graph.  To see this, we fix these vertices $i,j$, let $U_1:= \N(i)$, and for $m\in \mathbb{N}, m>1,$ let $U_m := \N(U_{m-1})$.  Since $G$ is a connected graph, if for any $n \in \mathbb{N}$ is true that  
$$ \bigcup_{i=1}^{n}U_i \subseteq V_1,$$
then $V_1 =V$. Otherwise, there exist $q \in \mathbb{N}$ such that  $\bigcup_{i=1}^{q}U_i \nsubseteq V_1$. Let $\ell \in \left(  \bigcup_{i=1}^{q} U_i  \right)   \cap V_2$. Then there is a $t \in \{1,\ldots,q\}$ such that $\ell \in U_t \cap V_2.$
Note that as $\ell \in U_t $ there is a path $\gamma=(i_{0},i_1, \ldots, i_{t-1},i_{t})$ of size $t$ connecting $i$ to $\ell$; i.e. $i_0=i$ and $i_t=\ell$. If  $t$ is even then $\deg(\ell)= \deg(i)$ and then $G$ is $\deg(1)$-regular. If $t$ is odd, we consider the path $\gamma_1=(j,i,i_1, \ldots, i_{t-1},i_{t})$ connecting $j$ to $\ell$, which has size even so $\deg(j)=\deg(\ell)$, but $\deg(j)=\deg(1)$, and therefore $G$ is a $\deg(1)$-regular graph. The same argument holds by assuming the existence of a pair of vertices $i,j \in V_2$ such that $i \in \N(j)$. 

\end{proof}

For the rest of the paper, we adopt the notation $f_{\pi}$ for a map between evolution algebras, with the same natural basis, defined by \eqref{eq:isofunc}. Even if $\mathcal{A}_{RW}(G) \cong \mathcal{A}(G)$ it is important to note that not every map defined as in \eqref{eq:isofunc} is an isomorphism, as we illustrate in the following example. 

\begin{exa}\label{exa:cycle}
Let $C_5$ the cycle graph or circular graph with $5$ vertices (see Fig. \ref{FIG:cycle}). 
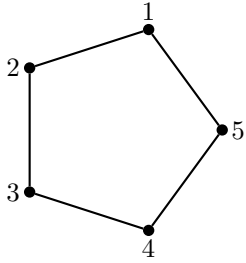
\begin{figure}[!h]
\begin{center}

\begin{tikzpicture}[scale=.7]
\GraphInit[vstyle=Simple]
    \SetGraphUnit{1}
    \tikzset{VertexStyle/.style = {shape = circle,fill = black,minimum size = 2.5pt,inner sep=1.5pt}}
  \begin{scope}[xshift=10cm]
\grCycle[prefix=a,RA=2]{5}%
\node[above] at (a1) {$1$};
\node[left] at (a2) {$2$};
\node[left] at (a3) {$3$};
\node[below] at (a4) {$4$};
\node at (2.3,0) {$5$};
\end{scope}
\end{tikzpicture}
\caption{Cycle graph $C_5$}
\end{center}
\label{FIG:cycle}
\end{figure}

Consider the evolution algebras induced by $C_5$, and by the random walk on $C_5$, respectively. That is, consider the evolution algebras whose natural basis is $\{e_1,e_2,e_3,e_4,e_5\}$ and relations are:

\smallskip
$$
\begin{array}{cc}
\mathcal{A}(C_5): \left\{
\begin{array}{l}
 e_1^2= e_2 + e_5,\\[.2cm]
 e_i^2=e_{i-1}+e_{i+1},   i\in \{2,3,4\},\\[.2cm]
 e_5^2=e_1 + e_{4}, \\[.2cm]
 e_i \cdot e_j =0,  i\neq j.
\end{array}\right.
&
\mathcal{A}_{RW}(C_5): \left\{
\begin{array}{l}
 e_1^2= \frac{1}{2}\,e_2 +\frac{1}{2}\, e_5,\\[.2cm]
 e_i^2=\frac{1}{2}\,e_{i-1}+\frac{1}{2}\,e_{i+1},   i\in \{2,3,4\},\\[.2cm]
 e_5^2=\frac{1}{2}\,e_1 +\frac{1}{2}\, e_{4}, \\[.2cm]
 e_i \cdot e_j =0,  i\neq j.
\end{array}\right.
\end{array}
$$

\vspace{.5cm}
\noindent
Note that $C_5$ is a $2$-regular graph, then by \cite[Theorem 3.2(i)]{PMP} $\mathcal{A}_{RW}(C_5) \cong \mathcal{A}(C_5)$ as evolution algebras. However, we shall see that not all map $f_{\pi}$,  with  $\pi \in S_5$, is an isomorphism. Indeed, let $\pi$ is given by


\begin{equation}\label{eq:piexa}
\pi:= \left(
\begin{array}{ccccc}
1 & 2 & 3 & 4 & 5\\[.2cm]
3& 2 & 1& 4 & 5
\end{array}\right).
\end{equation}
We shall verify that $f_{\pi}:\mathcal{A}_{RW}(C_5)\longrightarrow \mathcal{A}(C_5)$ defined by \eqref{eq:isofunc} is not an isomorphism. In order to do it, it is enough to note that

$$f_{\pi}(e_1 ^2) = \frac{1}{2}f_{\pi}\left(e_2 + e_5\right) = \frac{1}{2} \left(\alpha_2 e_{2} + \alpha_5 e_5\right),$$
while

$$f_{\pi}(e_1)\cdot f_{\pi}(e_1)= \alpha_1^2 e_3^2 = \alpha_1^2 \left(e_2 + e_4\right).$$

\vspace{.5cm}
\noindent
Therefore $f_{\pi}(e_1 ^2) \neq f_{\pi}(e_1)\cdot f_{\pi}(e_1)$.

\end{exa}


\begin{prop}\label{prop:pi}Let $G$ be a graph and let $A=(a_{ij})$ be its adjacency matrix. Assume $f_{\pi}:\mathcal{A}_{RW}(G)\longrightarrow \mathcal{A}(G)$ is an isomorphism defined as in \eqref{eq:isofunc}, i.e. 
\begin{equation*}
f_{\pi}(e_i) = \alpha_i e_{\pi(i)},\;\;\;\text{ for all }i\in V,
\end{equation*}
where $\alpha_i \neq 0$, $i\in V$, are scalars and $\pi \in {\color{black}S_{V}}$. Then $\pi$ satisfies 
\begin{equation}\label{eq:condpi}
a_{i\pi^{-1}(k)}\,\alpha _{\pi^{-1}(k)}=\deg(i)\,\alpha_i^2\,a_{\pi(i)k} , \text{  for } i,k \in V.
\end{equation}

\end{prop}
\begin{proof} Since $f_{\pi}$ is an homomorphism we have that 
$$ f_{\pi}(e_i^2)=f_{\pi}(e_i)\cdot f_{\pi}(e_i)=\alpha_{i}^2 e^{2}_{\pi(i)}= \alpha_{i}^{2} \sum_{k=1}^{|V|} a_{\pi(i)k} e_k , \text { for }  i\in V.$$
On the other hand 
$$f_{\pi}(e_{i}^{2})=f_{\pi}\left(\sum_{k=1}^{|V|}\left(\frac{a_{ik}}{\deg(i)}\right)e_k\right)= \sum_{k=1}^{|V|} \left(\frac{a_{ik}}{\deg(i)}\right) \alpha_{k}e_{\pi(k)}.$$
Then $$a_{i\pi^{-1}(k)}\,\alpha _{\pi^{-1}(k)}=\deg(i)\,\alpha_i^2\,a_{\pi(i)k},$$
for any $i,k \in V$, where $\pi^{-1} \in {\color{black}S_{V}}$ denotes the inverse of $\pi$, i.e. $\pi^{-1}(j)=i$ if, and only if, $\pi(i)=j$, for any $i,j\in V$. We notice that either in the case $|V|=\infty$ the previous sums are summations of a finite number of terms. This is because we are considering locally finite graphs.

\end{proof}

\begin{exa}

Let $C_5$ be the cycle graph considered in Example \ref{exa:cycle}, and let $f_{\pi}:\mathcal{A}_{RW}(C_5)\longrightarrow \mathcal{A}(C_5)$, where $\pi$ is given by \eqref{eq:piexa}. Taking $i=1$ and $k=4$ we have on one hand $a_{\pi(1) 4}=a_{34}=1$, while, on the other hand, $a_{1 \pi^{-1}(4)}=a_{14}=0$. This is enough to see that there exist no sequence of non-zero scalars $(\alpha_i)_{i\in V}$ such that \eqref{eq:condpi} holds. Therefore, by Proposition  \ref{prop:pi},  $f_{\pi}$ it is not an isomorphism. On the other hand, a straightforward calculation shows that the element of $S_5$ given by
    $$\sigma:= \left(
\begin{array}{ccccc}
1 & 2 & 3 & 4 & 5\\[.2cm]
2& 3 & 4& 5 & 1
\end{array}\right),$$
satisfies \eqref{eq:condpi}, provided $\alpha_i =1/2$ for any $i\in V$. {\color{black}Moreover it is possible to check that} $f_{\sigma}:\mathcal{A}_{RW}(C_5)\longrightarrow \mathcal{A}(C_5)$ defined for $i\in V$ by $f_{\sigma}(e_i) =(1/2) e_{\sigma(i)}$ is an isomorphism. 
\end{exa}

\subsection{Isomorphisms for the case of {\color{black}finite} non-singular graphs and proof of Theorem \ref{theo:criterio}}

\begin{prop}\label{theo:principal}
Let $G$ be a non-singular graph with $n$ vertices and let $A=(a_{ij})$ be its adjacency matrix. If $f:\mathcal{A}_{RW}(G)\longrightarrow \mathcal{A}(G)$ is an homomorphism, then either $f$ is the null map or $f$ is an isomorphism defined by
\begin{equation*} 
f(e_i) = \alpha_i e_{\pi(i)},\;\;\;\text{ for all }i\in V,
\end{equation*}
where $\alpha_i \neq 0$, $i\in V$, are scalars and $\pi$ is an element of the  symmetric group $S_n$.
\end{prop}

\begin{proof}
Let $f:\mathcal{A}_{RW}(G)\longrightarrow \mathcal{A}(G)$ an homomorphism such that 
\begin{equation}
\nonumber f(e_i)=\sum_{k=1}^{n} t_{ik}e_k, \,\,\, \text{for any } i\in V,
\end{equation}
where the $t_{ik}$'s are scalars. Then $f(e_i)\cdot f(e_j)=0$ for any $i\neq j$, which implies
$$
0=\sum_{k\in V}t_{ik} t_{jk} e_k^2 =\sum_{k\in V} t_{ik} t_{jk} \left(\sum_{r\in V} a_{kr} e_r\right)=\sum_{r\in V} \left(\sum_{k\in V} t_{ik} t_{jk} a_{kr}\right) e_r.
$$
This in turns implies, for any $r\in V$,
$$\sum_{k\in V} t_{ik} t_{jk} a_{kr} =0.$$
In other words we have, for $i\neq j$, $A^T
\begin{bmatrix}
t_{i1}t_{j1} & t_{i2}t_{j2} & \cdots & t_{in}t_{jn}
\end{bmatrix}^{T} = \begin{bmatrix}
0 & 0 & \cdots & 0
\end{bmatrix}^{T}$, where $B^T$ denotes the transpose of the matrix $B$. As the adjacency matrix $A$ is non-singular then
\begin{equation}\label{eq:tij}
\nonumber t_{ik}t_{jk}=0, \text{ for any }i,j,k\in V \text{ with }i\neq j.
\end{equation}
Thus for any fixed $k\in V$ we have $t_{ik} =0$ for all $ i\in V$, or there exists at most one $i:=i(k)\in V$ such that $t_{ik}  \not = 0$ and  $t_{jk} =0$ for all $j  \not = i$. Then 
\begin{equation}\label{eq:supp}
(\supp f(e_i) ) \cap (\supp f(e_j)) =\emptyset, \text{ for }i\neq j, \text{ and }\displaystyle \cup_{i\in V} (\supp f(e_i) ) = V,
\end{equation}
where $\supp f(e_i)=\{j\in V: t_{ij}\neq 0\}$. In what follows we consider two cases.\\

\noindent
{\bf Case 1.} For any $k\in V$ there exists $i\in V$ such that $t_{ik}  \not = 0$ and  $t_{jk} =0$ for all $j  \not = i$. In this case the sequence $(t_{ij})_{i,j\in V}$ contains only $n$ scalars different from zero. Assume that there exists $i\in V$ such that $t_{i j_1}\neq 0$ and $t_{i j_2}\neq 0$ for some $j_1, j_2 \in V$. This implies the existence of $m\in V$ such that $f(e_m) =0$, which in turns implies $f(e_m)\cdot f(e_m)=0$. On the other hand, as $f(e_m)\cdot f(e_m) = f(e_m^2)$, we have 
$$0=f(e_m^2) =f\left(\sum_{\ell \in V}\left(\frac{a_{m\ell}}{k_m}\right) e_{\ell}\right)=\sum_{\ell \in V} \left(\frac{a_{m\ell}}{k_m}\right) f(e_{\ell}).$$
We can use \eqref{eq:supp} to conclude $f(e_{\ell})=0$ for any $\ell$ such that $a_{m\ell}=1$. In other words, for any $\ell \in \mathcal{N}(m)$ it holds that $f(e_\ell)=0$. This procedure may be repeated, now for any $\ell \in \mathcal{N}(m)$, i.e. we can prove for any $v\in\mathcal{N}({\ell})$ that $f(e_v)=0$. As we are dealing with a connected graph, this procedure may be repeated until to cover all the vertices of $G$, and therefore we can conclude that $f(e_{i})=0$ for any $i\in V$, which is a contradiction. Therefore, for any $i\in V$ there exists only one $j:=j(i)$ such that $t_{ij}\neq 0$. Hence $f$ it must to be defined as $$f(e_i) = \alpha_i e_{\pi(i)},\;\;\;\text{ for all }i\in V,$$ where the $\alpha_i$'s are scalars, and $\pi$ is an element of the  symmetric group $S_n$.\\

\noindent
{\bf Case 2.} Assume that there exist $k\in V$ such that $t_{ik} =0$ for all $ i\in V$. Then the sequence $(t_{ij})_{i,j\in V}$ contains at most $n-1$ scalars different from zero, which implies the existence of $\ell \in V$ such that $f(e_{\ell})=0$. By applying similar arguments as the ones of Case 1 we conclude that $f(e_i)=0$ for any $i\in V$ and therefore $t_{ij}=0$ for any $i,j \in V$. Thus $f$ is the null map.

\end{proof}


\begin{remark}
In Proposition \ref{theo:principal} we assume that the adjacency matrix $A$ is non-singular. We point out that this hypothesis is equivalent to the transition matrix, say $A_{RW}$, of the symmetric random walk on $G$ be non-singular. In fact, if we denote by $F_i$ the $i$-th row of $A$ then we can write 
$A^T=
\begin{bmatrix}
F_1 & F_2 & \cdots & F_n
\end{bmatrix}
$. Then $A_{RW}^T=
\begin{bmatrix}
(1/\deg(1)) F_1 & (1/\deg(2)) F_2 & \cdots &(1/\deg(n)) F_n \end{bmatrix}
$, and therefore $\det A_{RW} = \left(\deg(1)\times \deg(2) \times \cdots \times \deg(n)\right)^{-1} \, \det A$.
\end{remark}

\subsubsection{Proof of Theorem \ref{theo:criterio}}

Together, \cite[Theorem 3.2(i)]{PMP}, and Propositions \ref{theo:generalization}, \ref{theo:sufficient}, \ref{theo:principal} gain in interest if we realize that they provide a necessary and sufficient condition for the existence of isomorphisms in the case of non-singular graphs.
Indeed, assume that $G$ is a non-singular graph and notice that $\A_{RW}(G)\cong \A(G)$ implies, by Proposition \ref{theo:principal}, that the isomorphisms between $\A_{RW}(G)$ and $\A(G)$ are given by \eqref{eq:isofunc}. Then by Proposition \ref{theo:sufficient} we conclude that $G$ is a regular or a biregular graph. For the reciprocal, it is enough to apply Proposition \ref{theo:generalization} and \cite[Theorem 3.2(i)]{PMP}.

{\color{black}
\begin{remark}
In Conjecture \ref{conjecture} we claim that Theorem \ref{theo:criterio} should be true for the case of finite singular graphs. Indeed, we believe that the conjecture should be true for infinite graphs too. To see that we notice the Propositions \ref{theo:generalization} and \ref{theo:sufficient} hold for infinite graphs. Therefore a generalization in this direction should be focus on an extension of Proposition \ref{theo:principal} to deal with infinite non-singular adjacency matrices.   
\end{remark}
}
\section{Connection with the automorphisms of $\mathcal{A}(G)$}

The purpose of this sections is twofold. First we show that the problem of looking for the isomorphisms between $\mathcal{A}_{RW}(G)$ and $\mathcal{A}(G)$ is equivalent to the question of looking for the automorphisms of $\mathcal{A}(G)$, provided $G$ is a regular graph. Second, we use the previous comparison to revisit a result obtained by \cite{camacho/gomez/omirov/turdibaev/2013}, which exhibit the automorphism group of an evolution algebra. Then we give a better presentation of such result.

As usual we use  $\aut \mathcal{A}(G)$ to denote the  automorphism group of $\mathcal{A}(G)$.

\begin{prop}\label{prop:auto}
Let $G$ be a $d$-regular graph. Then any isomorphism $f:\mathcal{A}_{RW}(G)\longrightarrow \mathcal{A}(G)$ induces a $g\in \aut \mathcal{A}(G)$. Analogously, any $g\in \aut \mathcal{A}(G)$ induces an isomorphism $f:\mathcal{A}_{RW}(G)\longrightarrow \mathcal{A}(G)$.
\end{prop}

\begin{proof}
Assume that $f:\mathcal{A}_{RW}(G)\longrightarrow \mathcal{A}(G)$ is an isomorphism and consider $g:\mathcal{A}(G)\longrightarrow \mathcal{A}(G)$ such that $g(e_i)=d\, f(e_i)$ for any $i\in V$. If $i\neq j$ then $g(e_i)\cdot g(e_j) = 0$. On the other hand, for $i\in V$
$$g(e_i ^2)=g\left(\sum_{j\in V}a_{ij} e_j \right) = \sum_{j\in V} a_{ij} d f(e_j),$$
while
$$g(e_i)\cdot g(e_i) = d^2\,f(e_i) \cdot f(e_i) = d^2\, f(e_i^2) = d^2\, f\left(\sum_{j\in V} \left(\frac{a_{ij}}{d} \right)e_j\right) = \sum_{j\in V} a_{ij} d f(e_j).$$
Thus $g$ is an automorphism of $ \mathcal{A}(G)$. The other assertion may be proved in analogous way by considering $f(e_i)=d^{-1}\, g(e_i)$ for any $i\in V$.

\end{proof}


The correspondence described in the previous proposition allow us to state the following result.

\begin{prop}\label{prop:autograph}
Let $G$ be a non-singular regular graph with $n$ vertices, and let $\mathcal{A}(G)$ be its associated evolution algebra. Then $\aut \mathcal{A}(G) \subseteq \{g_{\pi}: \pi \in S_{n}\}$.
 \end{prop}

\begin{proof}
Let $g \in \aut \mathcal{A}(G)$. By the proof of Proposition \ref{prop:auto}, there exists an isomorphism $f:\mathcal{A}_{RW}(G)\longrightarrow \mathcal{A}(G)$ such that $f(e_i) := (1/d)\,g(e_i)$, for any $i\in V$. On the other hand, as $G$ is a non-singular graph we have by Proposition \ref{theo:principal} that $f(e_i)=\alpha_i e_{\pi(i)}$, where the $\alpha_i$'s are scalars and $\pi$ is an element of the symmetric group $S_n$. Therefore $g=g_\pi$ and the proof is completed.  
\end{proof}

In \cite[Proposition 3.1]{camacho/gomez/omirov/turdibaev/2013} it has been stated that for any evolution algebra $E$ with a non-singular matrix of structural constants it holds that $\aut E=\{g_\pi:\pi \in S_n\}$. Example \ref{exa:cycle} shows that if $E:=\mathcal{A}(C_5)$ (so $\det A = 2$), then $\aut E\subsetneq \{g_\pi:\pi \in S_n\}$, which contradicts the equality stated by \cite{camacho/gomez/omirov/turdibaev/2013}. The mistake behind their result is in the proof. Indeed, although the authors assume correctly that an automorphism $g$ should verify $g(e_i\cdot e_j)=g(e_i)\cdot g(e_j)$, they only check this equality when $i\neq j$. When one check also the equality for $i=j$ one can obtain the condition that $\pi$ must satisfy in order to be an automorphism. This is the spirit behind our Proposition \ref{prop:pi}. The same arguments of our proof lead to the following version of \cite[Proposition 3.1]{camacho/gomez/omirov/turdibaev/2013}.

\begin{prop}
Let $E$ be an evolution algebra with natural basis $\{e_i:i\in V\}$, and a non-singular matrix of structural constants $C=(c_{ij})$. Then
$$\aut E=\{g_{\pi}:\pi \in S_n \text{ and }c_{i\pi^{-1}(k)}\,\alpha _{\pi^{-1}(k)}=\alpha_i^2\,c_{\pi(i)k} , \text{  for any } i,k \in V\}.$$
\end{prop}


\section*{Acknowledgments}

P.C. was supported by CNPq (grant number 235081/2014-0). P.M.R was supported by FAPESP (grant numbers 2016/11648-0, 17/10555-0) and CNPq (grant number 304676/2016-0).


\bigskip

\end{document}